\documentclass[12pt]{article}
\usepackage[russian]{babel}
\usepackage{amsmath, amssymb, amsfonts,amsthm,comment,graphicx,amscd}
\usepackage[unicode,pdftex]{hyperref}
\usepackage{comment}

\newtheorem{theorem}{Теорема}
\newtheorem{lemma}{Лемма}
\newtheorem{proposition}{Предложение}
\theoremstyle{definition}
\newtheorem{problem}{Задача}

\theoremstyle{remark}
\newtheorem*{remark}{Замечание}

\newcommand{\FF}{\widetilde{F}}

\title{Многочлены Чебышёва, их замечательные свойства и связь с числами Каталана}
\author{Андрей Рябичев\footnote{\href{https://mailto:ryabichev@179.ru}{ryabichev@179.ru}}, Константин Щербаков}
\date{}

\begin{document}

\maketitle

\begin{abstract}
Основной результат статьи говорит о том, что формальный степенной ряд,
равный отношению двух соседних многочленов Чебышёва,
после некоторой перенормировки приближает производящую функцию чисел Каталана.
Мы приводим доказательство этого результата с использованием цепных дробей,
а также, для полноты и замкнутости изложения,
напоминаем всё необходимое о многочленах Чебышёва и формальных степенных рядах.
\end{abstract}

\section{Введение}

Обозначим через $T_n(x)$ $n$-ый многочлен Чебышёва первого рода
(см.\ определение ниже, в \S\ref{s:chebyshov-def}).
Пусть $c_n$ --- $n$-ое число Каталана, то есть $c_0=c_1=1$, $c_2=2$, $c_3=5$, \ldots

\begin{theorem}\label{th:otnoshenye}
При $n>1$ в степенных рядах
$$y\cdot\frac{T_n\big(\frac1{2y}\big)}{T_{n-1}\big(\frac1{2y}\big)} \ \ \ \text{\it и}\ \ \ \textstyle1-c_0y^2-c_1y^4-c_2y^5-\ldots$$
совпадают младшие коэффициенты до степени $y^{2n-4}$ включительно.
\end{theorem}

Эта теорема, по-видимому, впервые была доказана совсем недавно \cite{a-b-sh}
(см.\ также обсуждение в \cite{b-sh}).
Мы приведём доказательство, похожее по смыслу, но немного более наглядное и эффектное технически.

Также мы постараемся напомнить все необходимые определения и доказать свойства объектов, с которыми будем работать.
А именно, в \S\ref{s:chebyshov-def} мы вводим многочлены Чебышёва, в \S\ref{s:fps} напоминаем определение и основные свойства формальных степенных рядов,
а в \S\ref{s:bliz}-\ref{s:pws} обсуждаем, какой смысл для них может иметь <<сходимость>> бесконечных цепных дробей.
Теорема~\ref{th:otnoshenye} доказывается в \S\ref{s:sootn-catalan}.

\subsubsection*{Благодарности}
Авторы признательны
А.\,В.\,Устинову
за идею доказательства теоремы~\ref{th:otnoshenye},
сообщённую в процессе обсуждении статьи \cite{r-s}.
Также авторы благодарны слушателям лекции 17 февраля 2024 в рамках Дня математика в 179 школе,
перед которыми это доказательство впервые было представлено.

\section{Многочлены Чебышёва и их свойства}\label{s:chebyshov-def}


Прежде чем определять многочлены Чебышёва, мы опишем задачу, в процессе решения которой они исторически возникли.
Назовём {\it уклонением} функции $f:[-1;1]\to\mathbb{R}$ число
$$\|f\|:=\max_{x\in[-1;1]}|f(x)|.$$
Эта буквально величина того, насколько функция $f(x)$ отклоняется от нуля на отрезке $[-1;1]$.

\begin{theorem}\label{th:otklonenye}
Если $f(x)$ --- приведённый многочлен степени $n>0$, то $\|f\|\ge\frac1{2^{n-1}}$.
Более того, существует единственный приведённый многочлен $p(x)$ степени $n$, такой что $\|p\|=\frac1{2^{n-1}}$.
\end{theorem}

Все $n$ корней этого многочлена $p(x)$ будут распределены по отрезку $[-1;1]$,
а локальные максимумы будут на высоте $\pm\frac1{2^{n-1}}$,
как это кажется необходимым, <<чтобы обеспечить минимальное отклонение>>.
Неожиданным образом, ответ в явном виде помогает получить следующая конструкция.

\begin{theorem}\label{th:chebyshov}
Для каждого $n$ существует многочлен $T_n(x)$ такой, что $T_n(\cos t)=\cos nt$.
\end{theorem}

Многочлены $T_n$ называются {\it многочленами Чебышёва первого рода}.
Теорему~\ref{th:chebyshov} легко вывести, например, из формулы Муавра.
Мы используем иную стратегию, заодно получив рекуррентную формулу для $T_n$.

\begin{proof}[Доказательство теоремы~\ref{th:chebyshov}]
Будем действовать по индукции.
Положим $T_1(x)=x$ и \linebreak$T_2(x)=2x^2-1$.
Для произвольного $n>1$ заметим, что
$$\cos((n+1)t)+\cos((n-1)t)=2\cos(nt)\cos t.$$
По предположению индукции $\cos((n-1)t)=T_{n-1}(\cos t)$ и $\cos(nt)=T_n(\cos t)$.
Отсюда
\begin{equation}\label{eq:reccurenta}
T_{n+1}(x)=2xT_n(x)-T_{n-1}(x),
\end{equation}
то есть $T_{n+1}(x)$ действительно является многочленом от $x$.
\end{proof}

Помимо важного рекуррентного соотношения (\ref{eq:reccurenta}), мы получаем следующий факт:

\begin{problem}
Старший коэффициент в $T_n(x)$ равен $2^{n-1}$.
\end{problem}

Оказывается, искомый многочлен в теореме~\ref{th:otklonenye}
задаётся как $p(x)=\frac1{2^{n-1}}T_n(x)$.
Доступное доказательство читатель может найти в статьях \cite{gashkov} и \cite{tabachnikov},
мы не будем его приводить.

Многочлены Чебышёва $T_n(x)$ обладают массой интересных свойств (см., например, \cite{zelevinsky}).
Упомянем одно из них: многочлены Чебышёва являются {\it коммутирующими}
(что среди многочленов, вообще говоря, встречается достаточно редко, см.\ например
\cite[\S10-12]{prasolov-shvartsman}).

\begin{problem}
Докажите, что $T_m(T_n(x))=T_n(T_m(x))$.
\end{problem}

\section{Разложение отношения $T_n/T_{n-1}$ в цепную дробь}

Пусть $g_n(x):=\frac{T_n(x)}{T_{n-1}(x)}$.
Отметим, что $g_n$ является {\it рациональной функцией}, но не многочленом.
Поделив обе части равенства (\ref{eq:reccurenta}) на $T_n(x)$, получаем соотношение

\begin{equation}\label{eq:recc-cep}
g_{n+1}(x) = 2x - \frac1{g_n(x)}.
\end{equation}

Далее, подставляя в выражение (\ref{eq:recc-cep}) для $n+1$ выражение (\ref{eq:recc-cep}) для $n$,
в него выражение (\ref{eq:recc-cep}) для $n-1$, и так далее, мы получаем разложение $g_{n+1}(x)$ в цепную дробь:

\begin{equation}\label{eq:cep-primer}\notag
g_1(x) = x, \ \ \
g_2(x) = 2x - \frac1x, \ \ \
g_3(x) = 2x - \cfrac1{2x-\frac1x}, \ \ \
g_4(x) = 2x - \cfrac1{2x - \cfrac1{2x-\frac1x}}.
\end{equation}

В этом выражении удобно сделать замену $x=\frac1{2y}$ и домножить его на $y$.
Положим
$G_n(y):=y\cdot g_n\big(\frac1{2y}\big)$.
Например, $G_4(y)$ становится такой двухэтажной цепной дробью:

\begin{equation}\label{eq:cep-primer4}\notag
G_4(y) =
y\cdot\left( \frac1y - \cfrac 1{\frac1y - \cfrac1{\frac1y-2y}}\right)=
1 - \cfrac y{\frac1y - \cfrac1{\frac1y-2y}} =
1 - \cfrac{y^2}{1 - \cfrac{y}{\frac1y-2y}} =
1 - \cfrac{y^2}{1 - \cfrac{y^2}{1-2y^2}}.
\end{equation}

Для произвольного $n$
общий вид этой цепной дроби
легко выводится из (\ref{eq:recc-cep}) по индукции:

\begin{proposition}\label{prop:G}
При $n>1$ рациональная функция $G_n(y)$ равна следующей $(n-2)$-этажной цепной дроби:
\begin{equation}\label{eq:cep-gn}
G_n(y) =
1 - \cfrac{y^2}{1 - \cfrac{y^2}{1-\cfrac{y^2}{\ddots-\cfrac{y^2}{1-2y^2}}}}.
\end{equation}
\end{proposition}

\section{Формальные степенные ряды}\label{s:fps}

Напомним, {\it формальным степенным рядом от переменной $x$}
называется выражение вида $a_0+a_1x+a_2x^2+\ldots$
Формальные степенные ряды можно складывать и умножать так же, как многочлены:
для рядов $a_0+a_1x+a_2x^2+\ldots$ и $b_0+b_1x+b_2x^2+\ldots$
коэффициент при $x^k$ в их сумме будет равен $a_k+b_k$,
а в их произведении будет равен $a_kb_0+a_{k-1}b_1+\ldots+a_0b_k$.

Иногда формальные степенные ряды можно и делить.
{\it Отношением} $(a_0+a_1x+\ldots)/(b_0+b_1x+\ldots)$
назовём ряд $h(x)$ такой что $h(x)\cdot(b_0+b_1x+\ldots)=a_0+a_1x+\ldots$
Например, $1/(1+x)=1-x+x^2-x^3+\ldots$,
чтобы убедиться в этом, достаточно
домножить обе части равенства
на $1+x$ и раскрыть скобки.
Однако, отношение определено не всегда:

\begin{problem}\label{prob:powerseries-divide}
Пусть в обозначениях выше $a_0\ne0$.
Тогда $h(x)$ существует если и только если $b_0\ne0$.
Более того, такой ряд $h(x)$ единственный.
\end{problem}

Чтобы избавиться от этого условия и уметь делить любые ряды,
достаточно разрешить конечному числу мономов иметь отрицательную степень.
А именно, {\it рядом Лорана} будем называть формальное выражение вида
$a_{m}x^m+a_{m+1}x^{m+1}+\ldots$, где $m$ --- произвольное целое число.
Сумма, произведение и отношение рядов Лорана определяется аналогично.
Отметим, что любой ряд Лорана можно представить как некоторый формальный степенной ряд,
умноженный на $\frac1{x^k}$ для некоторого неотрицательного $k$.

\begin{problem}\label{prob:loran-divide}
Для любых ненулевых рядов Лорана их отношение
существует и определено однозначно.
\end{problem}

Как видно из последнего утверждения, множество рядов Лорана является {\it полем частных} для кольца формальных степенных рядов.
В дальнейшем мы в основном будем пользоваться лишь формальными степенными рядами 
(для которых важно понятие {\it делимости}), 
а ряды Лорана возникнут только в \S\ref{quad-sootn}.

Подробнее с теорией формальных степенных рядов и их применениями в комбинаторике
читатель может познакомиться, например,
в \cite[гл.\,VIII]{vilenkin} или в \cite{lando}.

\section{Лемма о близости}
\label{s:bliz}

Наша цель --- придать смысл (когда это возможно) бесконечным цепным дробям, составленным из формальных степенных рядов.
Это позволит работать с <<пределом>> рядов $G_n(y)$, определённых выше.
Оказывается, последовательность $G_n$ <<стабилизируется>> в следующем смысле:

\begin{lemma}[лемма о близости]\label{l:bliz}
    Пусть даны две
    $n$-этажные
    цепные дроби вида
    \begin{equation}\label{eq:A(x)}
    A(x)=1 + \cfrac{b_1(x)}{1 + \cfrac{b_2(x)}{\ddots+\cfrac{b_n(x)}{1 + P(x)}}}
    \end{equation}
    и
    \begin{equation}\label{eq:B(x)}
    B(x)=1 + \cfrac{b_1(x)}{1 + \cfrac{b_2(x)}{\ddots+\cfrac{b_n(x)}{1 + Q(x)}}},
    \end{equation}
    где $b_i(x)$, $P(x)$, $Q(x)$ ---
    формальные степенные ряды, имеющие нулевой свободный член.
    Тогда разложения $A(x)$ и $B(x)$ в степенной ряд совпадают в младших членах до $x^n$ включительно.
\end{lemma}
\begin{proof}
    Будем действовать индукцией по $n$.
    База для $n=0$ следует из того, что оба ряда имеют свободный член $1$.
    %
    Докажем переход от $n-1$ к $n$.
    Для этого запишем
    выражения (\ref{eq:A(x)}) и (\ref{eq:B(x)})
    в виде
    \begin{equation}\label{eq:sum1}
        A(x)=1+\cfrac{b_1(x)}{1+\widetilde{A}(x)}=1+b_1(x)\cdot\left(1-\widetilde{A}(x)+\widetilde{A}(x)^2-\widetilde{A}(x)^3+\ldots\right)
    \end{equation}
    и
    \begin{equation}\label{eq:sum2}
        B(x)=1+\frac{b_1(x)}{1+\widetilde{B}(x)}=1+b_1(x)\cdot\left(1-\widetilde{B}(x)+\widetilde{B}(x)^2-\widetilde{B}(x)^3+\ldots\right).
    \end{equation}
    Здесь ряды $\widetilde{A}(x)$ и $\widetilde{B}(x)$ не содержат свободного члена.

    По предположению индукции, ряды $1+\widetilde{A}(x)$ и $1+\widetilde{B}(x)$ совпадают в членах вплоть до $x^{n-1}$.
    Тогда в рядах в скобках из выражений (\ref{eq:sum1}) и (\ref{eq:sum2})
    после приведения подобных
    совпадут коэффициенты при степенях до $x^{n-1}$.
    Следовательно, поскольку $b_1(x)$ делится на $x$, сами $A(x)$ и $B(x)$ совпадут до $x^{n}$ включительно, что доказывает переход индукции.
\end{proof}

\section{Бесконечные цепные дроби формальных степенных рядов}
\label{s:pws}

Рассмотрим бесконечную цепную дробь вида
\begin{equation}\label{eq:goodfrac}
R(x) := \alpha_0 + \cfrac{b_1(x)}{1 + \cfrac{b_2(x)}{1+\cfrac{b_3(x)}{1 + \ldots}}},
\end{equation}
где $b_i(x)$ ---
формальные степенные ряды с нулевым свободным членом, $\alpha_0$ --- произвольное вещественное число.
Для
$R(x)$ назовём
{\it $n$-й подходящей дробью} следующую $n$-этажную цепную дробь:
\begin{equation}\label{eq:goodfracfinite}
R_n(x) := \alpha_0 + \cfrac{b_1(x)}{1+\cfrac{b_2(x)}{1+\cfrac{b_3(x)}{\ddots+\cfrac{b_{n-1}(x)}{1+b_n(x)}}}},
\end{equation}
Отметим, что $R_n(x)$ корректно задаёт формальный степенной ряд ---
чтобы получить его, нужно $n$ раз воспользоваться утверждением задачи \ref{prob:powerseries-divide}.

Из леммы о близости, разложения в ряд $n$-й и $(n+1)$-й подходящих дробей совпадают до степени $x^n$ включительно.
Значит, для каждого $k$ коэффициент при $x^k$ будет одинаковым
во всех $R_n(x)$ при $n\ge k$.
Это наблюдение можно сформулировать следующим образом.

\begin{theorem}
Бесконечная цепная дробь вида (\ref{eq:goodfrac})
корректно задаёт формальный степенной ряд.
\end{theorem}

{\it Единственность} цепной дроби, задающей формальный степенной ряд,
будет иметь место при дополнительном условии.


\begin{problem}
    Докажите, что для любого степенного ряда существует единственная равная ему
    конечная или бесконечная цепная дробь вида (\ref{eq:goodfracfinite}) или (\ref{eq:goodfrac}),
    в которой все $b_i(x)$ являются одночленами, то есть имеют вид $\beta_ix^{k_i}$
    для некоторых $\beta_i\in\mathbb{R}\setminus\{0\}$ и $k_i>0$.
\end{problem}

\begin{remark}
В более привычной ситуации, когда бесконечная цепная дробь составлена из чисел,
обычно рассматривают дроби с единицами в числителе на каждом этаже --- такая дробь всегда будет сходящейся
(см., например, \cite{khinchin}).

Для цепной дроби, составленной из формальных степенных рядов,
мы наблюдаем противоположный эффект ---
<<сходимость>> дроби,
доказанная в лемме о близости,
обусловлена тем, что на каждом этаже 
числитель делится на $x$.
В качестве же <<элементов>> цепной дроби
(слагаемых, не входящих в дробь на очередном этаже)
мы берём единицы,
чтобы деление было всегда определено.
\end{remark}

\section{Степенной ряд для $G_n(y)$}

Определим
следующую бесконечную цепную дробь:
\begin{equation}\label{eq:g-infty}
G(y) := 1 - \cfrac{y^2}{1 - \cfrac{y^2}{1-\cfrac{y^2}{1 - \ldots}}}
\end{equation}

\begin{proposition}\label{prop:low-coeff}
    Ряд цепной дроби $G(y)$ совпадает с рядом $G_n(y)$ в младших членах до $y^{2n-4}$ включительно.
\end{proposition}

\begin{proof}
    Поскольку $y$ входит в $G(y)$ и $G_n(y)$ только в чётных степенях,
    на эти выражения можно смотреть как на цепные дроби от $t=y^2$.
    Напомним, что $G_n(t)$ имеет $n-2$ этажа
    и отличается от $(n-2)$-й подходящей дроби для $G(t)$ лишь в последнем знаменателе.
    Следовательно, по лемме о близости,
    ряды $G_n(t)$ и $G(t)$
    совпадают до степени $t^{n-2}$.
    Значит, $G_n(y)$ и 
    $G(y)$ при разложении в ряд совпадают до степени $y^{2n-4}$.
\end{proof}

\section{Производящая функция для чисел Каталана}
\label{s:sootn-catalan}

Напомним, {\it $n$-ое число Каталана} $c_n$ равно количеству правильных скобочных структур из $n$ пар скобок.
Числа Каталана имеют много эквивалентных определений и массу интересных свойств
(см., например, \cite{gardner} или \cite[\S2.5]{lando}).
Нам потребуется одно из них, легко выводящееся из приведённого определения.

\begin{problem}\label{prob:recurrenta-catalan}
Имеется следующее соотношение: $c_{n+1}=c_0c_n+c_1c_{n-1}+\ldots+c_nc_0$.
\end{problem}

Рассмотрим {\it производящую функцию} для чисел Каталана,
то есть формальный степенной ряд
$C(x):=c_0+c_1x+c_2x^2+\ldots$

\begin{proposition}\label{prop:sootn-catalan}
Формальный степенной ряд $C(x)$ удовлетворяет уравнению
\begin{equation}\label{eq:sootn-catalan}
C(x)-1=x\cdot\big(C(x)\big)^2.
\end{equation}
\end{proposition}

\begin{proof}
Это прямо следует из определения умножения формальных степенных рядов и утверждения задачи~\ref{prob:recurrenta-catalan}.
\end{proof}

Уравнение (\ref{eq:sootn-catalan}) можно переписать в виде
$1-1/C(x)=x\cdot C(x)$,
далее преобразовать к виду
$1/C(x)=1-x\cdot C(x)$,
и в итоге получить
$C(x)=1/(1-x\cdot C(x))$.

Заметим, что последнее равенство устанавливает <<самоподобие>> производящей функции $C(x)$:
\begin{equation}\label{eq:gener-catalan}
C(x)=
\cfrac1{1-x\cdot C(x)}=
\cfrac1{1-\cfrac{x}{1-x\cdot C(x)}}=
\cfrac1{1-\cfrac{x}{1-\cfrac{x}{1-x\cdot C(x)}}}=\ldots
\end{equation}

По лемме о близости и из представления (\ref{eq:gener-catalan}) видно, что ряд $C(x)$, совпадает со значением бесконечной цепной дроби
\begin{equation}\label{eq:cepn-catalan}
\cfrac1{1-\cfrac{x}{1-\cfrac{x}{1-\ldots}}}.
\end{equation}
Действительно,
$k$-этажная цепная дробь, аналогичная (\ref{eq:gener-catalan}),
по лемме о близости совпадает с (\ref{eq:cepn-catalan}) до степени $x^{k-1}$ включительно
(степень на единицу меньше количества этажей, поскольку числитель первого этажа не делится на $x$).

\begin{proof}[Доказательство теоремы~\ref{th:otnoshenye}]
Сравнивая (\ref{eq:g-infty}) и (\ref{eq:cepn-catalan}),
мы получаем соотношение $G(y)=1-y^2\cdot C(y^2)$.
Чтобы завершить доказательство теоремы~\ref{th:otnoshenye},
остаётся воспользоваться
определением $G_n(y)=y\cdot T_n\big(\frac1{2y}\big)/T_{n-1}\big(\frac1{2y}\big)$
и предложением~\ref{prop:low-coeff}.
\end{proof}


\section{Дополнение}

\subsection{Цепные дроби и квадратичные соотношения}
\label{quad-sootn}

Отметим, что соотношение (\ref{eq:sootn-catalan}) является квадратным уравнением с коэффициентами в рядах Лорана относительно $C(x)$.
Поэтому кроме корня $C(x)$, являющегося производящей функцией последовательности чисел Каталана, у него есть второй корень, равный по теореме Виета $\widetilde{C}(x)=1/x-C(x)$. Оба этих ряда будут удовлетворять равенствам (\ref{eq:gener-catalan}).

Ясно, что значение бесконечной цепной дроби (\ref{eq:cepn-catalan}) должно быть корнем квадратного уравнения (\ref{eq:sootn-catalan}), причём последовательность подходящих дробей будет сходится именно к корню, не содержащему отрицательных степеней $x$. <<Исчезновение>> второго корня связанно с тем, что лемма о близости не будет верна, если заменить в ней степенные ряды на ряды Лорана
(а потому
в цепочке (\ref{eq:gener-catalan}) <<стабилизации коэффициентов>> не происходит).

В то же время,
мы могли бы выразить $C(x)$ как $1/x - 1/(x\cdot C(x))$.
Из этого выражения получилась бы бесконечная цепная дробь,
значение которой совпадает со вторым корнем уравнения (\ref{eq:sootn-catalan}).

Отметим, что всякая периодическая цепная дробь соответствует некоторой квадратичной иррациональности, и наоборот, по всякой квадратичной иррациональности можно построить соответствующую периодическую бесконечную цепную дробь. Подробнее про эту связь в случае приближения вещественных чисел можно прочитать в \cite[гл.\,II,~\S10]{khinchin}.



\subsection{Цепная дробь для отношения соседних чисел Фибоначчи}

Напомним ещё одно похожее рассуждение.
Пусть $F_0,F_1,\ldots$ --- последовательность чисел Фибоначчи.
Можно переписать определяющее их соотношение $F_{n+1}=F_n+F_{n-1}$
в виде
\begin{equation}\notag
\frac{F_{n+1}}{F_n} =
1+\cfrac1{\frac{F_n}{F_{n-1}}} =
1+\cfrac1{1+\cfrac1{\frac{F_{n-1}}{F_{n-2}}}} = \ldots
\end{equation}
Отсюда следует известное соотношение

\begin{equation}\notag
\lim_{n\to\infty}\frac{F_{n+1}}{F_n} =
1+\cfrac1{1+\cfrac1{1+\ldots}} = \varphi,
\end{equation}
где $\varphi$ --- золотое сечение, то есть положительный корень уравнения $1+\frac1\varphi=\varphi$.

Наше доказательство теоремы~\ref{th:otnoshenye} по сути аналогично ---
используя рекуррентное соотношение для многочленов Чебышёва,
мы раскладываем их отношение в периодическую цепную дробь,
а она в свою очередь равна известной производящей функции $C(x)$,
поскольку удовлетворяет тому же квадратному уравнению (\ref{eq:sootn-catalan}).

\end{document}